\documentclass[11pt]{article}
\usepackage[numbers,sort&compress]{natbib}
\usepackage{enumerate}
\usepackage{amscd}
\usepackage{amsmath}
\usepackage{latexsym}
\usepackage{amsfonts}
\usepackage{amssymb}
\usepackage{amsthm}
\usepackage{verbatim}
\usepackage{mathrsfs}
\usepackage{enumerate}
\usepackage{hyperref}

\theoremstyle{plain}\newtheorem{definition}{Definition}[section]
\theoremstyle{definition}\newtheorem{theorem}{Theorem}[section]
\theoremstyle{plain}\newtheorem{lemma}[theorem]{Lemma}
\theoremstyle{plain}\newtheorem{coro}[theorem]{Corollary}
\theoremstyle{plain}
\theoremstyle{remark}\newtheorem{remark}{Remark}[section]
\usepackage{xcolor}

\newcommand{\Div}{\mathrm{div}\,}
\newcommand{\B}{\Big}

\newcommand{\be}{\begin{equation}}
\newcommand{\ee}{\end{equation}}
 \newcommand{\ba}{\begin{aligned}}
 \newcommand{\ea}{\end{aligned}}

  \newcommand{\f}{\frac}
    
  \newcommand{\ben}{\begin{enumerate}}
   \newcommand{\een}{\end{enumerate}}

\newcommand{\Rmnum}[1]{\expandafter\@slowromancap\romannumeral #1@}

\allowdisplaybreaks

\numberwithin{equation}{section}
\UseRawInputEncoding
\begin{document}
\title{A general sufficient criterion for energy conservation in the   Navier-Stokes system }
\author{ Yanqing Wang\footnote{ College of Mathematics and Information Science, Zhengzhou University of Light Industry, Zhengzhou, Henan  450002,  P. R. China Email: wangyanqing20056@gmail.com}
\;~ and\, \,
Yulin Ye\footnote{Corresponding author. School of Mathematics and Statistics,
Henan University,
Kaifeng, 475004,
P. R. China. Email: ylye@vip.henu.edu.cn   }}
\date{}
\maketitle
\begin{abstract}
In this paper, we derive an energy conservation criterion based on  a combination of velocity and its gradient for the weak solutions of both the homogeneous
incompressible Navier-Stokes equations and the
 general compressible Navier-Stokes equations.
For the incompressible case,  this  class implies   most   known  corresponding results on periodic domain via either  the velocity  or its gradient   including  the famous Lions'  energy  conservation criterion obtained in \cite{[Lions]}.
  For the  compressible case,  this   helps us to extend the previously known criteria for the energy conservation  of weak solutions from the incompressible  fluid  to   compressible flow and improve the recent results due to   Nguyen-Nguyen-Tang in \cite[Nonlinearity 32 (2019)]{[NNT]} and Liang in \cite[Proc. Roy. Soc. Edinburgh Sect. A (2020)]{[Liang]}.
     \end{abstract}
\section{Introduction}
\label{intro}
\setcounter{section}{1}\setcounter{equation}{0}
The homogeneous    Navier-Stokes equations describing the motion of   incompressible fluid in
three-dimensional space
 read
  \be\left\{\ba\label{NS}
  &  v _{t}-\Delta v +\Div(  v\otimes v)+\nabla
\pi=0,\quad(x,t)\in\Omega\times(0,+\infty),\\
&\Div v=0,\\
&v_{0}=v(x,0).
\ea\right.\ee
 Here, $v$ stands for the velocity field of the flow and
 $\pi$ represents the  pressure of the fluid, respectively. The initial datum satisfies $\Div v_0=0$.
Usually, one considers the Navier-Stokes equations on the periodic domain  $(\Omega=\mathbb{T}^3)$, on smooth bounded domain $\Omega$ with Dirichlet  boundary condition or on the whole space
 $(\Omega=\mathbb{R}^3)$.

 It is well known that Leray-Hopf weak solutions of the Navier-Stokes equations \eqref{NS} obeys the energy inequality
 $$
 \|v(T)\|_{L^{2}(\Omega)}^{2}+2 \int_{0}^{T}\|\nabla v\|_{L^{2}(\Omega)}^{2}ds\leq \|v_0\|_{L^{2}(\Omega)}^{2},
 $$
 rather than energy equality.
 The first attempts to determine sufficient conditions implying energy conservation of Leray-Hopf weak solutions in the homogeneous incompressible Navier-Stokes equations were given by  Lions \cite{[Lions]} and Prodi \cite{[Prodi]} provided that the criterion  \begin{equation}\label{lions}v\in L^4(0,T;L^4(\Omega))\end{equation} is satisfied. Then Serrin \cite{[Serrin1963]} showed the energy conservation by giving a criterion in a scaling invarant space, that is
$$v\in L^p(0,T; L^q(\Omega)),\ \text{with}\  \frac{2}{p}+\frac{d}{q}\leq 1 \ \text{and}\  q\geq d,$$
where $d$ is the spatial dimension. However, the weak solution which satisfy the given criterion will immediately become a classical one.
Later, Shinbrot extended Lions' condition in \cite{[Lions]}
for energy conservation to
\be\label{Shinbrot}
v\in L^{p}(0,T;L^{q}(\Omega)) ~\text{ with} ~\f{2}{p}+
\f{2}{q}=1, q\geq 4.
\ee
It is worth remarking that condition \eqref{Shinbrot} is weaker than \eqref{lions} when the dimension $d\geq 4$ and more importantly, it is true regardless of the dimension of the underlying space. On the other hand, the energy conservation condition \eqref{Shinbrot} can be replaced by
\be\label{TBV}
v\in L^{p}(0,T;L^{q}(\Omega)) \text{ with}\  \f{1}{p}+
\f{3}{q}=1, 3<q< 4;
\ee
which   was recently obtained by   Beirao da Veiga-Yang in \cite{[BY]}. Using the  Fourier methods, Cheskidov-Friedlander-Shvydkoy \cite{[CFS2010]}
gave the following sufficient condition for energy conservation (here $A$ denotes the Stokes operator associated to the Dirichlet boundary conditions) that
$$ A^{5/12}v\in L^3(0,T; L^2(\Omega)),$$
in fact, this criterion is equivalent in terms of scaling to $v\in L^3(0,T; L^\frac{9}{2}(\Omega)).$
Very recently, Berselli-Chiodaroli  \cite{[BC]} and Zhang \cite{[Zhang]} obtained energy equality  via the following condition,
\be\label{bcz}
\nabla v\in L^{p}\left(0, T ; L^{q}\left(\Omega\right)\right),
\frac{1}{p}+\frac{3 }{q}=2,   \frac{3 }{2}<q<\frac{9}{5}~~\text{ or}~~
\frac{1}{p}+\frac{6}{5 q}=1,   \frac{9}{5} \leq q.\ee
It is worth remarking that the domain $\Omega$ in  most the    aforementioned  conditional results  is the smooth bounded. The first objective of this paper is to show  the following sufficient criterion for the weak solutions keeping the energy  of the Navier-Stokes equations
on the  periodic domain $\mathbb{T}^{3},$
\be\label{wy}
  v\in L^{\f{2p}{p-1}} (0,T;L^{\f{2q}{q-1}}(\mathbb{T}^{3}) ) ~~  \text{and}~~ \nabla v \in L^{p} (0,T; L^{q}(\mathbb{T}^{3}) ).
 \ee
Surprisingly, this result covers the corresponding  results of \eqref{Shinbrot}, \eqref{TBV} and \eqref{bcz}  on periodic domain and further discussion will be found in Remark \ref{rem1.2} and \ref{rem1.3}. Indeed, we will also prove this class for the following  compressible Navier-Stokes equations with
degenerate viscosities and general pressure law (GNS)
\be\left\{\ba\label{GNS}
&\rho_{t} + \Div (\rho v)   =0,   \\
&(\rho v)_{t} +\Div(\rho v\otimes v)+\nabla p (\rho)- \Div(\nu(\rho)\mathbb{D}v)-\nabla(\mu(\rho)\Div v)= 0,\\
\ea\right.\ee
with the initial data
\begin{equation}\label{gns1}
	\begin{aligned}
		\rho(0,x)=\rho_0(x)\ \text{and} \ (\rho v)(0,x)=\rho_0(x)v_0(x),\ \ x\in \Omega,
	\end{aligned}
\end{equation}
where the unknown functions $\rho$ and $v$ denote the density of the fluid and velocity of the fluid, respectively; $\mathbb{D}v=\frac{1}{2}(\nabla v\otimes \nabla v^{T} )$ stands for the  stain tensor; The general pressure $0\leq p(\rho)\in C^{1}(0,\infty)$ with $p^{'}(\cdot)>0$ and the viscosity coefficients $\nu(\rho),\mu(\rho):(0,\infty)\rightarrow [0,\infty)$ are continuous functions of density. We will
consider the case of bounded domain with periodic boundary conditions, namely $\Omega=\mathbb{T}^d$, where $d\geq 2$ is the dimension of the domain.
It should be noted that when $\nu(\rho)\equiv\nu, \mu(\rho)\equiv \mu$ and  $p(\rho)=\rho^\gamma $ with $\gamma >1$, then (GNS) will reduce to the classical  isentropic compressible Navier-Stokes equations (ICNS):
\be\left\{\ba\label{ICNS}
&\rho_t+\Div (\rho v)=0, \\
&(\rho v)_{t} +\Div(\rho v\otimes v)+\nabla
\rho^\gamma -\nu\Delta v- \mu \nabla\text{div\,}v=0,
\ea\right.\ee
 and when $\nu(\rho)=\rho , \mu(\rho)\equiv 0$ and $p(\rho)=\rho^\gamma $ with $\gamma >1$, then (GNS) reduces to the compressible Navier-Stokes equations with degenerate viscosity and $\gamma $-pressure law (CNSD) as follows:
\be\left\{\ba\label{CNSD}
&\rho_{t} + \Div (\rho v)   =0,   \\
&(\rho v)_{t} +\Div(\rho v\otimes v) +\nabla \rho^\gamma- \Div(\rho \mathbb{D}v) = 0.
\ea\right.\ee
The global existence of weak solutions which satisfy the energy inequality has already been known, see P. L. Lions \cite{Lions} and Feireisl-Novotn\'y-Petzeltov\'a \cite{FNP2001} for (ICNS) with constant viscosity coefficients case, Vassure-Yu \cite{VY2016} for (CNSD) and  Li-Xin \cite{LX2015} for (GNS) with degenerate viscosity coefficients case, but the regularity and even the uniqueness of  weak solutions are still open problems. Since the weak solutions satisfy the energy inequality rather than  equality due to the basic a priori estimates and the lack of regularity, this anomalous dissipation of the energy opens a possibility for an energy sink other than the natural viscous
dissipation, however, such a property of real fluids is not expected to exist
physically, so it is therefore a famous problem to consider the sufficient criterion for the energy conservation of  weak solutions. Roughly speaking, this addresses the question how much regularities are needed for a weak solution to conserve energy, which is also involving the uniqueness of the weak solutions. On the other hand, energy conservation is also one aspect of the Onsager's conjecture  in the context of homogeneous incompressible Euler equations in \cite{[Onsager]}, in which Onsager conjectured that the kinetic energy is globally conserved for H\"older continuous solutions with the exponent greater than $\frac{1}{3}$, while
an energy dissipation phenomenon occurs for H\"older continuous solutions with the exponent less than $\frac{1}{3}$. For the positive part,  the  milestone  work is due to Constantin-E-Titi
\cite{[CWT]}, in which it was proved that the energy of 3D incompressible
Euler equations is conserved for every weak solution in $L^{3}(0,T;
B_{3,\infty}^{\alpha})$ with $\alpha>1/3$. On the other hand, Isett resolved
the \textquotedblleft negative\textquotedblright part of Onsager's conjecture
for 3D incompressible Euler equations in \cite{[Isett2016]}, where he proved
that for any $\alpha<\frac{1}{3}$ there is a nonzero weak solution to the
incompressible Euler equations in the class $v\in C^\alpha_{t,x}$ and $p \in
C^{2\alpha}_{t,x}$ such that $v$ is identically $0$  outside a finite time
interval. In particular, the solution $v$ fails to conserve the energy. We
refer the reader to \cite{[ABCDJK2021],[ADSW],[BT], [BV2019],
	[CY],[CCFS],[EGSW],[NNT1],[NNT2],[BDSV]} for
recent
progress in
this direction.

  Compared with the incompressible Navier-Stokes equations, due to the stronger nonlinearity, the energy conservation of the weak solutions is more challenging and hence the results are few. When the density is strictly away from vacuum,  for the weak solutions to the general  compressible models \eqref{GNS},
Nguye-Nguye-Tang \cite{[NNT]} established the   Shinbrot-type criterion and showed that if the weak solutions satisfied
\be\ba\label{NNT}
&0<  c_{1}\leq\rho\leq c_{2}<\infty,  v\in L^{\infty}(0,T; L^{2}
(\mathbb{T}^{3})), \nabla v\in L^{2}(0,T; L^{2}(\mathbb{T}^{3})),   \\&\sup_{t\in(0,T)}\sup_{|h|<\varepsilon}|h|^{-\f12}
\|\rho(\cdot+h,t)-\rho(\cdot,t)\|_{L^{2}(\mathbb{T}^{3})}<\infty,\\
&v\in L^{p}(0,T;L^{q}(\mathbb{T}^{3})) \text{ with}~~ \left\{ \ba &\f{2}{p}+
\f{2}{q}=1, q\geq 4,\\&\f{1}{p}+
\f{3}{q}=1, 3<q< 4,\ea\right.
\ea\ee
then the energy of  weak solutions is globally conserved, which means the energy equality holds for any $t\in [0,T]$.
It is worth noting that though the part $\f{1}{p}+
\f{3}{q}=1, 3<q< 4$  was not mentioned in \cite{[NNT]}, it is   a direct consequence from  interpolation and $v\in L^{4}(0,T;L^{4}(\mathbb{T}^{3}))$ (see \cite{[BY]} and the corresponding proof in Theorem \ref{the1.1}).

Later, Liang \cite{[Liang]}  derived a energy conservation criterion via the gradient of velocity for   isentropic Navier-Stokes equations (ICNS) under the following condition
\be \label{lns} \ba
&0<  c_{1}\leq\rho\leq c_{2}<\infty, v\in L^{\infty}(0,T; L^{2}(\mathbb{T}^{3})),       \nabla v\in L^{2}(0,T; L^{2}(\mathbb{T}^{3})), \\
&\nabla v \in L^{p}\left(0, T ; L^{s}(\mathbb{T}^{3})\right) ~\text{with}~ \left\{ \ba
\frac{1}{p}+\frac{3}{s}<2, \ \ \ \ \ \ & \frac{3}{2}<s<\frac{9}{5}, \\
\frac{5}{p}+\frac{6}{s}<5,\ \ \ \  \ \  & \frac{9}{5} \leq s\leq3, \\
\frac{1}{p}+\frac{2}{s+2}<1, \ \ & 3<s<\infty ,
\ea\right.\ea\ee
then the energy of weak solutions is locally conserved, which means the energy equality holds in the sense of distribution in $(0,T)$.

When the density may contain vacuum,   in the spirit of well-known Shinbrot's criterion in \cite{[Shinbrot]},
  Yu \cite{[Yu2]} showed that if a weak solution  $(\rho, v)$ of \eqref{CNSD}  or \eqref{ICNS} satisfies
\be\ba\label{yu}
\sqrt{\rho}v\in L^{\infty}(0,T;L^{2}(\Omega)),\sqrt{\rho}\nabla v\in L^{2}(0,T;L^{2}(\Omega)),\\
0\leq\rho\leq c<\infty, \ \ \nabla\sqrt{\rho} \in L^{\infty}(0,T;L^{2}(\Omega))\\
v\in L^{p}(0,T;L^{q}(\Omega)) ~~\f1p+\f1q\leq \f5{12}   \ \text{ and}\   q\geq6,
\ea
\ee
then the energy is globally conserved. Recently, for equations (\ref{ICNS}), Chen-Liang-Wang-Xu \cite{[CLWX]} obtained the energy balance   in a bounded domain with physical boundaries under the following condition
\be\label{clwx}\ba
& \sqrt{\rho}v\in L^{\infty}(0,T;L^{2}(\Omega)),\sqrt{\rho}\nabla v\in L^{2}(0,T;L^{2}(\Omega)),\\
& 0\leq \rho\leq c<\infty, \nabla\sqrt{\rho} \in L^{\infty}(0,T;L^{2}(\Omega)),\\
& v\in L^{p}(0,T;L^{q}(\Omega)),p \geq4, q\geq6.
\ea\ee

As mentioned above, the energy conservation for fluid equations addresses the question how much regularities are needed for a weak solution to conserve energy, which will help us further to consider the uniqueness and regularity of the weak solutions. However, up to now, the related results on energy conservation of weak solutions for compressible Navier-Stokes equations are less satisfactory than incompressible ones.  For example, compared with the Shinbrot's condition \eqref{Shinbrot} for the incompressible Navier-Stokes equations, the criterion \eqref{NNT} obtained by Nguye-Nguye-Tang in \cite{[NNT]} requires additional constraint that $\sup\limits_{t\in(0,T)}\sup\limits_{|h|<\varepsilon}|h|^{-\f12}
\|\rho(\cdot+h,t)-\rho(\cdot,t)\|_{L^{2}(\mathbb{T}^{3})}<\infty$. Note that,   due to the blow-up criteria only via  the density for strong solutions to the Cauchy problem of compressible isentropic Navier-Stokes equations in $\mathbb{R}^{3}$, under the assumptions on the coefficients of viscosity established in \cite{[SWZ],[WZ]} it implies the bound of
density to the 3-D compressible Navier-Stokes
equations yields strong solutions and the strong solutions are expected to meet
 energy conservation.
 Based on this, the second objective of this paper is  to remove  $\eqref{NNT}_{2}$ to obtain the persistence of energy. Moreover, when we see the criterion obtained by  Liang in  \cite{[Liang]}, on the one hand, the restrictions on the indexes $p$ and $s$ are \textquotedblleft subcritical\textquotedblright other than \textquotedblleft critical\textquotedblright and are stronger when $s>3$ compared with the result \eqref{bcz} for incompressible case. On the other hand, the criterion obtained in \cite{[Liang]} only implies the energy conserved \textquotedblleft locally \textquotedblright not \textquotedblleft globally\textquotedblright. Hence, our third objective of this paper is to improve the criterion via the gradient of velocity and to show $\eqref{NNT}_{1}$ and  \eqref{bcz} guarantee the energy equality in system \eqref{GNS} globally.

Before stating the main results, we introduce the definition of the weak solutions.
 \begin{definition}
 	A pair ($\rho,v$) is called a weak solution to \eqref{GNS} with initial data ($\rho_{0},v_{0}$) if ($\rho,v$) satisfy
 	\begin{enumerate}[(i)]
 		\item equations \eqref{GNS} hold in $\mathcal{D}'(0,T;\mathbb{T}^d)$ and
 		\begin{equation}\label{1.13}
 			P(\rho ), \rho |v|^2\in L^\infty(0,T;L^1(\mathbb{T}^d)),\ \ \ \nabla v\in L^2(0,T;L^2(\mathbb{T}^d)),
 		\end{equation}
 		\item[(ii)] $\rho(\cdot,t)\rightharpoonup \rho_0$ in $\mathcal{D}'(\Omega)$ as $t\rightarrow 0$, i.e.
 		\begin{equation}\label{1.14}
 		\lim_{t\rightarrow 0}\int_{\mathbb{T}^d}\rho (x,t)\varphi(x)dx=\int_{\mathbb{T}^d}\rho_0(x)\varphi(x)dx,
 		\end{equation}
 		for every test function $\varphi\in C_0^\infty(\mathbb{T}^d).$
 		\item[(iii)]
 		$(\rho v)(\cdot,t)\rightharpoonup \rho_0 v_0$ in $\mathcal{D}'(\mathbb{T}^d)$ as $t\rightarrow 0$ i.e.
 		\begin{equation}\label{1.15}
 		\lim_{t\rightarrow 0}\int_{\mathbb{T}^d}(\rho v)(x,t)\psi(x) dx=\int_{\mathbb{T}^d} (\rho_0 v_0)(x)\psi(x) dx,
 		\end{equation}
 		for every test vector field $\psi\in C_0^\infty (\mathbb{T}^d)^d$.
 		\item[(iv)]
 		the energy inequality holds
 		\begin{equation}\label{1.16}
 		\begin{aligned}
 		\mathcal{E}(t) +\int_0^T\int_{\mathbb{T}^d}&\left[\nu(\rho)|\mathbb{D} v|^2+\mu(\rho)|\Div v|^2 \right]dxdt\leq \mathcal{E}(0),
 		\end{aligned}\end{equation}
 		where $\mathcal{E}(t)=\int_{\mathbb{T}^d}\left[\frac{1}{2}\rho |v|^2+P(\rho) \right] dx$ and $P(\rho)=\rho \int_{1}^{\rho} \frac{p(z)}{z^2}dz.$
 	\end{enumerate}
 \end{definition}
We formulate our first result as follows:
\begin{theorem}\label{the1.2}
 For any dimension $d\geq2$, let $(\rho, v)$  be a weak solution  to the general compressible Navier-Stokes equations \eqref{GNS}. Assume that $1<p,q<\infty$ and
 \be\left\{\ba\label{u}
&0<  c_{1}\leq\rho\leq c_{2}<\infty,   v\in L^{\infty}(0,T; L^{2}(\mathbb{T}^{d})),   \nabla v\in L^{2}(0,T; L^{2}(\mathbb{T}^{d})),   \\
& v\in L^{\f{2p}{p-1}}(0,T; L^{\f{2q}{q-1}}(\mathbb{T}^{d})), \ \nabla v\in L^{p}(0,T; L^{q}(\mathbb{T}^{d})) \ and\ \sqrt{\rho_0}v_0\in L^{2+\delta}\ for \ any\ \delta>0,
\ea\right.\ee
then the energy of weak solutions is globally conserved, that is, for any $t\in [0,T]$
\begin{equation}\label{EI}
	\begin{aligned}
		\mathcal{E}(t) +\int_0^T\int_{\mathbb{T}^d}&\left[\nu(\rho)|\mathbb{D} v|^2+\mu(\rho)|\Div v|^2 \right]dxdt= \mathcal{E}(0),
\end{aligned}\end{equation}
where $\mathcal{E}(t)=\int_{\Omega}\left[\frac{1}{2}\rho |v|^2+P(\rho) \right] dx$ and $P(\rho)=\rho \int_{1}^{\rho} \frac{p(z)}{z^2}dz.$
\end{theorem}

\begin{remark}
We follow the path of \cite{[NNT]} to prove Theorem \ref{the1.2}.
The improvement of their condition \eqref{NNT}  are threefold. First,
Theorem \ref{the1.2} removed the additional restriction on the regularity of density. Second,  Theorem \ref{the1.2} not only covers their result \eqref{NNT} but also allow us to derive new criterion (  see the following corollary). Third, the regularity of pressure
$p(\rho)  $ is relaxed from $ C^{2}(0,\infty)$ in \cite{[NNT]} to $ C^{1}(0,\infty)$.
\end{remark}
\begin{remark}\label{rem1.2}
At first glance, energy conservation criteria \eqref{u} based on a combination of velocity and its gradient are more complicated than \eqref{vei} and \eqref{tivei}, however, \eqref{u} together with natural  energy    $v\in L^{\infty}(0,T; L^{2}(\mathbb{T}^{3})),   \nabla v\in L^{2}(0,T; L^{2}(\mathbb{T}^{3}))$ leads to \eqref{vei} and \eqref{tivei} in the following corollary.
\end{remark}
\begin{remark}\label{rem1.3}
By small modification of proof in Theorem \ref{the1.2}, the results in Theorem \ref{the1.2} also hold for homogenous incompressible Navier-Stokes equations \eqref{NS}, that is, $v\in L^{\f{2p}{p-1}}(0,T; L^{\f{2q}{q-1}}(\mathbb{T}^{3})) $ and $\nabla v\in L^{p}(0,T; L^{q}(\mathbb{T}^{3})) $ means the energy equality in the classical homogenous incompressible Navier-Stokes equations.
The special case  $p=q=2$ reduces to the famous Lions' energy  conservation criterion  \eqref{lions}.
As mentioned in latter remark, this result covers the \eqref{Shinbrot}-\eqref{bcz}, hence, roughly speaking,  this      unifies  the known energy  conservation criteria  via the velocity  and its gradient in incompressible Navier-Stokes equations. After we finished this paper, we learnt that  a special case that $p=3,q=9/5$ and away from $1/2$-H\"older continuous curve in time for general energy equality in the  homogeneous Navier-Stokes equations \eqref{NS}  in $\mathbb{R}^{3}$ was considered in \cite{[Shvydkoy]}.
\end{remark}
\begin{remark}
  The new
ingredient in the proof of this theorem is the application of the following inequality
\be\label{keyin}
\B\|\nabla\B(\f{(\rho v)^{\varepsilon}}{\rho^{\varepsilon}} \B)\B\|_{L^{p}(\mathbb{T}^{d})}
\leq C\|\nabla v\|_{L^{p}(\mathbb{T}^{d})}.
\ee
  This help us to pass the limit of pressure term only with the
positive bounded density, which removes the additional restriction of the density $\eqref{NNT}_2$ in \cite{[NNT]}. For the proof of \eqref{keyin}, we refer the readers to Lemma \ref{lem2.3} (see also \cite[page 7]{[Liang]}).
 \end{remark}
 \begin{remark}
  One can consider  Theorem \ref{the1.2} and Corollary \ref{the1.1} on smooth bounded domain.
Combining the framework for bounded domain in \cite{[NNT]} and the proof here, one only needs to deal with the boundary terms caused by integrations by parts. Fortunately,
these additional terms are
the lower order terms.
 \end{remark}
\begin{remark}
In dimension $d=2$, the Gagliardo-Nirenberg inequality
  guarantees that
$$\|v\|_{L^4(0,T;L^4(\mathbb{T}^2))}\leq C\|v\|_{L^\infty(0,T;L^2(\mathbb{T}^2))}^{\frac{1}{2}}\|\nabla v\|_{L^2(0,T;L^2(\mathbb{T}^2))}^{\frac{1}{2}}\leq C.$$
Therefore, according to Theorem \ref{the1.2},   the bounded density with positive lower bound and natural energy yield  the energy conservation of the weak solutions.
\end{remark}

Taking the natural energy of weak solutions into account, one  immediately  derives  the following
corollary.
 \begin{coro}\label{the1.1}
When the dimension $d=3$, if the  weak solutions  $(\rho, v)$ to the Navier-Stokes equation \eqref{GNS} satisfy one of the following two conditions
 \begin{enumerate}[(1)]
 \item $0<  c_{1}\leq\rho\leq c_{2}<\infty$ , $v\in L^{\infty}(0,T; L^{2}(\mathbb{T}^{3}))$, $ \nabla v\in L^{2}(0,T; L^{2}(\mathbb{T}^{3}))$ and $\sqrt{\rho_0}v_0\in L^{2+\delta}(\mathbb{T}^3)$ for any $\delta>0$,
 \be\label{vei}
v\in L^{p}(0,T;L^{q}(\mathbb{T}^{3}))  ~\text{with}~ \left\{ \ba &\f{2}{p}+
\f{2}{q}=1, q\geq 4,\\&\f{1}{p}+
\f{3}{q}=1, 3<q< 4;\ea\right. \ee
 \item $0<  c_{1}\leq\rho\leq c_{2}<\infty,$   $v\in L^{\infty}(0,T; L^{2}(\mathbb{T}^{3}))$  , $ \nabla v\in L^{2}(0,T; L^{2}(\mathbb{T}^{3}))$ and $\sqrt{\rho_0}v_0\in L^{2+\delta}(\mathbb{T}^3)$ for any $\delta>0$,
 \be\label{tivei}
   \nabla v\in L^{p}(0,T;L^{q}(\mathbb{T}^{3}))  ~\text{with}~\left\{ \ba &\frac{1}{p}+\frac{3 }{q}=2,  \  \frac{3 }{2}<q<\frac{9}{5},\\&\frac{1}{p}+\frac{6}{5 q}=1,   \frac{9}{5} \leq q,\ea\right.
 \ee
 then the energy is globally conserved, that is, for any $t\in [0,T]$,
 \begin{equation}
 	\begin{aligned}
 		\mathcal{E}(t) +\int_0^T\int_{\mathbb{T}^3}&\left[\nu(\rho)|\mathbb{D} v|^2+\mu(\rho)|\Div v|^2 \right]dxdt= \mathcal{E}(0),
 \end{aligned}\end{equation}
 where $\mathcal{E}(t)=\int_{\mathbb{T}^3}\left[\frac{1}{2}\rho |v|^2+P(\rho) \right] dx$ and $P(\rho)=\rho \int_{1}^{\rho} \frac{p(z)}{z^2}dz.$
  \end{enumerate}
\end{coro}
\begin{remark}
Compared with result \eqref{NNT} obtained by Nguye-Nguye-Tang \cite{[NNT]}, conditions \ref{vei} only required the density is bounded  from below and above. Hence, result \eqref{vei} is an improvement of \eqref{NNT} in \cite{[NNT]}.
\end{remark}
\begin{remark}
We extend the  energy conservation  criteria  \eqref{Shinbrot}-\eqref{bcz} from incompressible Navier-Stokes equations to
general compressible Navier-Stokes equations with no vacuum.
\end{remark}
\begin{remark}
In
contrast with \eqref{lns},
the generalization   in \eqref{tivei} is threefold: first, to improve the
corresponding results in \eqref{tivei}; second, to consider the more general equations; third, we can get the energy conservation up to the initial time $t=0$.
\end{remark}

\begin{remark}
It seems that
a new strategy for studying the  energy equality of
fluid equations
is to firstly establish a
conservation criterion based on  a combination of velocity and its gradient, which may
be applied to other incompressible and compressible fluid equations.
 A successful application can be found in \cite{[WYM]}.
\end{remark}

\begin{remark}
In the forthcoming work \cite{[YWW]}, the energy conservation criterion   for the weak solutions of  general compressible Navier-Stokes equations allowing  vacuum will be considered.
\end{remark}

Finally, as   \cite{[NNT]}, one can establish  the results  parallel to  Theorem \ref{the1.2} and Corollary \ref{the1.1} for the non-homogenous incompressible Navier-Stokes equations below
\be\left\{\ba\label{nhins}
&\rho_{t} + \Div (\rho v)   =0,   \\
&(\rho v)_{t} +\Div(\rho v\otimes v)- \Div(\nu(\rho)\mathbb{D}v) +\nabla\pi = 0,\\
&\Div v=0,\\
&(\rho,v)|_{t=0}=(\rho_0,u_0),
\ea\right.\ee
we leave this to the interested readers.

The remainder of this paper is organized as follows. Section 2 is devoted  to the
auxiliary lemmas involving mollifier  and the key inequality \eqref{keyin}.
In section 3, we first
 present the proof of Theorem \ref{the1.2}. Then, based on Theorem \ref{the1.2},
 we
 complete the proof of Corollary \ref{the1.1}.
\section{Notations and some auxiliary lemmas} \label{section2}

First, we introduce some notations used in this paper.
 For $p\in [1,\,\infty]$, the notation $L^{p}(0,\,T;X)$ stands for the set of measurable functions on the interval $(0,\,T)$ with values in $X$ and $\|f(\cdot,t)\|_{X}$ belonging to $L^{p}(0,\,T)$. The classical Sobolev space $W^{k,p}(\mathbb{T}^d)$ is equipped with the norm $\|f\|_{W^{k,p}(\mathbb{T}^d)}=\sum\limits_{\alpha =0}^{k}\|D^{\alpha}f\|_{L^{p}(\mathbb{T}^d)}$. The space  $C^{\infty}_{b}(\mathbb{T}^d)$ is the bounded smooth functions on $\mathbb{T}^d$. $c_1,c_2$ and $C$ are positive constants. For simplicity, we denote by $$\int_0^T\int_{\mathbb{T}^{d}} f(x,t)dxdt=\int_0^T\int f\ ~~\text{and}~~ \|f\|_{L^p(0,T;X )}=\|f\|_{L^p(X)}.$$

Let $\eta_{\varepsilon}:\mathbb{R}^{d}\rightarrow \mathbb{R}$ be a standard mollifier.i.e. $\eta(x)=C_0e^{-\frac{1}{1-|x|^2}}$ for $|x|<1$ and $\eta(x)=0$ for $|x|\geq 1$, where $C_0$ is a constant such that $\int_{\mathbb{R}^d}\eta (x) dx=1$. For $\varepsilon>0$, we define the rescaled mollifier $\eta_{\varepsilon}(x)=\frac{1}{\varepsilon^d}\eta(\frac{x}{\varepsilon})$. For any function $f\in L^1_{loc}(\Omega)$, its mollified version is defined as
$$f^\varepsilon(x)=(f*\eta_{\varepsilon})(x)=\int_{\mathbb{R}^d}f(x-y)\eta_{\varepsilon}(y)dy,\ \ x\in \Omega_\varepsilon,$$
where $\Omega_\varepsilon=\{x\in \Omega: d(x,\partial\Omega)>\varepsilon\}.$

We first recall the results involving the mollifier established in
\cite{[NNT]}.
\begin{lemma}(\cite{[NNT]})\label{lem2.1}
Suppose that $f\in L^{p}(0,T;L^{q}(\mathbb{T}^{d}))$. Then for any $\varepsilon>0$, there holds
\be\label{}
\|\nabla f^{\varepsilon}\|_{L^{p}(0,T;L^{q}(\mathbb{T}^{d}))}
\leq C\varepsilon^{-1}\| f\|_{L^{p}(0,T;L^{q}(\mathbb{T}^{d}))},
\ee
and, if $p,q<\infty$
$$
\limsup_{\varepsilon\rightarrow0} \varepsilon\|\nabla f^{\varepsilon}\|_{L^{p}(0,T;L^{q}(\mathbb{T}^{d}))}=0.
$$
Moreover, if $0<c_{1}\leq g\leq c_{2}<\infty$, then there holds, for any $\varepsilon>0$,
\be\ba
\B\|\nabla \f{f^{\varepsilon}}{g^{\varepsilon}}\B\|_{L^{p}(0,T;L^{q}(\mathbb{T}^{d}))}\leq C
\varepsilon^{-1}\|f\|_{L^{p}(0,T;L^{q}(\mathbb{T}^{d}))},
\ea\ee
and if $p,q<\infty$
\be
\limsup_{\varepsilon\rightarrow0} \varepsilon\B\|\nabla \f{f^{\varepsilon}}{g^{\varepsilon}}\B\|_{L^{p}(0,T;L^{q}(\mathbb{T}^{d}))}
=0.\ee
\end{lemma}
The next lemma with $p=q, p_{1}=q_{1}, p_{2}=q_{2}$ was  proved in \cite{[NNT]}.
We   generalize it by extending the integral  norms with different exponents in space and time.
\begin{lemma} \label{lem2.2}   Let $1\leq p,q,p_1,p_2,q_1,q_2\leq \infty$  with $\frac{1}{p}=\frac{1}{p_1}+\frac{1}{p_2}$ and $\frac{1}{q}=\frac{1}{q_1}+\frac{1}{q_2}$. Assume $f\in L^{p_1}(0,T;W^{1,q_1}(\mathbb{T}^d))$ and $g\in L^{p_2}(0,T;L^{q_2}(\mathbb{T}^d))$. Then for any $\varepsilon> 0$, there holds
		\begin{align} \label{fg'}
		\|(fg)^\varepsilon-f^\varepsilon g^\varepsilon\|_{L^p(0,T;L^q(\mathbb{T}^d))}\leq C\varepsilon \|f\|_{L^{p_1}(0,T;W^{1,q_1}( \mathbb{T}^d))}\|g\|_{L^{p_2}(0,T;L^{q_2}(\mathbb{T}^d))}.
		\end{align}
		Moreover, if $p_2,q_2<\infty$ then
		\begin{align}\label{limite'}
		\limsup_{\varepsilon \to 0}\varepsilon^{-1} \|(fg)^\varepsilon-f^\varepsilon g^\varepsilon\|_{L^p(0,T;L^q(\mathbb{T}^d))}=0.
		\end{align}
	\end{lemma}
 	\begin{proof}
		Thanks to the fact observed in    \cite{[CWT]} and the ideas in \cite{[Lions1]}, we know that
		\begin{align}\label{b14}
		(fg)^\varepsilon-f^\varepsilon g^\varepsilon=R^\varepsilon-(f^\varepsilon-f)(g^\varepsilon-g),
		\end{align}
		where
		\begin{align*}
		R^\varepsilon(x,t):=\int_{\mathbb{R}^d}\big(f\left(y,t\right)-f(x,t)\big) \big(g(y,t)-g(x,t)\big)\eta_\varepsilon(x-y)dy.
		\end{align*}
Using the triangle's inequality, it yields that
		\begin{align}\label{b15}
		\|(fg)^\varepsilon-f^\varepsilon g^\varepsilon\|_{L^q(\mathbb{T}^d) }\leq C\left(\|R^\varepsilon\|_{L^q(\mathbb{T}^d) }+\|(f-f^\varepsilon)(g-g^\varepsilon)\|_{L^q(\mathbb{T}^d)}\right).
		\end{align}	
Let $B(x,\varepsilon)=\{y\in \mathbb{T}^d; |x-y|<\varepsilon\}$, then by means of 		 H\"older's inequality  and direct computation, we see that
		\begin{equation} \label{b16}
		\begin{aligned}
	|R^\varepsilon|&\leq \int _{B(x,\varepsilon)}\frac{1}{\varepsilon^d}|f(y)-f(x)||g(y)-g(x)| dy\\
	&\leq C\left(\frac{1}{\varepsilon^d}\int_{B(x,\varepsilon)}|f(y)-f(x)|^{s_1}dy\right)^{\frac{1}{s_1}}\left(\frac{1}{\varepsilon^d}\int_{B(x,\varepsilon)}|g(y)-g(x)|^{s_2}dy\right)^{\frac{1}{s_2}}\\
	&\leq C\varepsilon\left(\frac{1}{\varepsilon^d}\int_{B(x,\varepsilon)}\int_{0}^{1}|\nabla f(x+(y-x)s)|^{s_1}ds dy\right)^{\frac{1}{s_1}}\left(\frac{1}{\varepsilon^d}\int_{B(x,\varepsilon)} |g(y)|^{s_2} dy +|g(x)|^{s_2}\right)^{\frac{1}{s_2}}\\
	&\leq C\varepsilon\left(\int_{B(0,1)}\int_0^1 |\nabla f(x+\omega\varepsilon s)|^{s_1}ds d\omega \right)^{\frac{1}{s_1}}\left(\int_{B(0,1)}|g(x+\omega\varepsilon)|^{s_2}d\omega +|g(x)|^{s_2}\right)^{\frac{1}{s_2}}\\
	&\leq C\varepsilon\left(\int_{\mathbb{R}^d}|
	\nabla f(x-z)|^{s_1}\int_0^1\frac{\bf{1}_{B(0,\varepsilon s)}(z)}{(\varepsilon s)^d}dsdz\right)^{\frac{1}{s_1}}\left(\int_{\mathbb{R}^d}|g(x-z)|^{s_2}\frac{\bf{1}_{B(0,\varepsilon)(z)}}{\varepsilon^d}dz+|g(x)|^{s_2}\right)^{\frac{1}{s_2}}\\
	&\leq C\left(|\nabla f|^{s_1}*J_\varepsilon\right)^{\frac{1}{s_1}}\left(|g|^{s_2}*J_{1\varepsilon}+|g(x)|^{s_2}\right)^{\frac{1}{s_2}},
		\end{aligned}
		\end{equation}
		where $s_1\leq q_1, s_2\leq q_2$ with $\frac{1}{s_1}+\frac{1}{s_2}=1$, $J_\varepsilon=\int_0^1 \frac{\bf{1}_{B(0,\varepsilon s)}}{(\varepsilon s)^d} ds\geq 0$, $J_{1\varepsilon}=\frac{\bf{1}_{B(0,\varepsilon)}}{\varepsilon^d}\geq 0$ and $\int_{\mathbb{R}^d}\int_0^1 \frac{\bf{1}_{B(0,\varepsilon s)}}{(\varepsilon s)^d} dsdz=\int_{\mathbb{R}^d}\frac{\bf{1}_{B(0,\varepsilon)}}{\varepsilon^d}dz=measure (B(0,1)).$

Then in view of  the Minkowski inequality, we  conclude that
\begin{equation}\label{b17}
\begin{aligned}
\|R^\varepsilon\|_{L^q}&\leq C\varepsilon\|\left(|\nabla f|^{s_1} *J_\varepsilon\right)^{\frac{1}{s_1}}\left(|g|^{s_2} *J_{1\varepsilon}+|g(x)|^{s_2}\right)^{\frac{1}{s_2}}\|_{L^q}\\
&\leq C\varepsilon\left[\|(|\nabla f|^{s_1}* J_\varepsilon)^{\frac{1}{s_1}}\|_{L^{q_1}}\left(\|(|g|^{s_2}*J_{1\varepsilon})^{\frac{1}{s_2}}\|_{L^{q_2}}+\|g\|_{L^{q_2}}\right)\right]\\
&\leq C\varepsilon\|\nabla f\|_{L^{q_1}}\|g\|_{L^{q_2}}.
\end{aligned}
\end{equation}		
Furthermore, one has
\begin{equation}\label{b18}
\begin{aligned}
&|(f^\varepsilon-f)(g^\varepsilon-g)|\\\leq& \int |(f(y)-f(x))|\eta_\varepsilon(x-y)dy\int |(g(y)-g(x))|\eta_\varepsilon(x-y)dy\\
\leq& C\varepsilon\left(\frac{1}{\varepsilon^d}\int_{B(x,\varepsilon)}\int_0^1|\nabla f(x+(y-x)s)|ds dy\right)\left(\frac{1}{\varepsilon^d}\int_{B(x,\varepsilon)}|g(y)-g(x)|dy\right)\\
\leq& C\varepsilon \left(\frac{1}{\varepsilon^d}\int_{B(x,\varepsilon)}\int_0^1|\nabla f(x+(y-x)s)|^{s_1}ds dy\right)^{\frac{1}{s_1}}\left(\frac{1}{\varepsilon^d}\int_{B(x,\varepsilon)}|g(y)-g(x)|^{s_2}dy\right)^{\frac{1}{s_2}}.
\end{aligned}
\end{equation}
Along the same lines of derivation of 	
  \eqref{b16} and \eqref{b17}, we arrive at
\begin{equation}\label{b19}
\|(f^\varepsilon-f)(g^\varepsilon -g)\|_{L^q}\leq C\varepsilon\|\nabla f\|_{L^{q_1}}\|g\|_{L^{q_2}}.
\end{equation}	
In combination with \eqref{b14}, \eqref{b17} and \eqref{b19}	and using the H\"older's inequality with respect to time, we can deduce the result \eqref{fg'}.

	 Furthermore, if $q_1, q_2<\infty$, let $\{g_n\}\in C_{b}^\infty (\mathbb{T}^d)$ with $ g_n\rightarrow g$ strongly in $L^{q_2}$. Thus, by  density arguments, we find that
	 \begin{equation}
	 \begin{aligned}
	 \|(fg)^\varepsilon-f^\varepsilon g^\varepsilon\|_{L^q}&\leq C\left\|(\left(f(g-g_n)\right)^\varepsilon+(f g_n)^\varepsilon-f^\varepsilon(g-g_n)^\varepsilon-f^\varepsilon g_n ^\varepsilon\|_{L^q}\right)\\
	 &\leq C \left(\|\left(f(g-g_n)\right)^\varepsilon-f^\varepsilon (g-g_n)^\varepsilon\|_{L^q}+\|(f g_n)^\varepsilon-f^\varepsilon g_n^\varepsilon\|_{L^q}\right)\\
	 &\leq C\left(\varepsilon\|\nabla f\|_{L^{q_1}}\|g-g_n\|_{L^{q_2}}+\varepsilon^2\|\nabla f\|_{L^{q_1}}\|\nabla g_n\|_{L^{q_2}}\right),
	 \end{aligned}
	 \end{equation}
which  means
	\begin{equation}
	\begin{aligned}
	\varepsilon^{-1}\|(fg)^\varepsilon-f^\varepsilon g^\varepsilon\|_{L^q}\leq C\left(\|\nabla f\|_{L^{q_1}}\|g-g_n\|_{L^{q_2}}+\varepsilon\|\nabla f\|_{L^{q_1}}\|\nabla g_n\|_{L^{q_2}}\right),
	\end{aligned}
	\end{equation}	
	hence, as $\varepsilon\rightarrow 0$ and $n\rightarrow \infty $ , we can obtain that
	\begin{equation}
	\begin{aligned}
	&\varepsilon^{-1}\|(fg)^\varepsilon-f^\varepsilon g^\varepsilon\|_{L^p(L^q)}\\
\leq&  C\left(\int_0^T\left(\|\nabla f\|_{L^{q_1}}\|g-g_n\|_{L^{q_2}}+\varepsilon\|\nabla f\|_{L^{q_1}}\|\nabla g_n\|_{L^{q_2}}\right)^p dt\right)^{\frac{1}{p}}\\
\leq& C\left(\int_0^T(\|\nabla f\|_{L^{q_1}}\|g-g_n\|_{L^{q_2}})^p dt\right)^{\frac{1}{p}}+C\varepsilon\left(\int_0^T (\|\nabla f\|_{L^{q_1}}\|\nabla g_n\|_{L^{q_2}})^p dt\right)^{\frac{1}{p}}\\
\leq &C\|\nabla f\|_{L^{p_1}(L^{q_1})}\|g-g_{n}\|_{L^{p_2}(L^{q_2})}+\varepsilon\|\nabla f\|_{L^{p_1}(L^{q_1})}\|\nabla g_n\|_{L^{p_2}(L^{q_2})}\rightarrow 0.
	\end{aligned}
	\end{equation}	
Then, we have completed the proof of Lemma \ref{lem2.2}.
	\end{proof}
The next lemma is the  key to remove $\eqref{NNT}_{2}$.
\begin{lemma}\label{lem2.3}
		Assume that $0<\underline{\rho}\leq \rho (x,t)\leq \overline{\rho}<\infty$ and $v\in W^{1,p}(\mathbb{T}^d)$ with $1\leq p\leq \infty$. Then
		\begin{equation}\label{b1}
\B\|\partial\left(\frac{(\rho v)^\varepsilon}{\rho ^\varepsilon}\right)\B\|_{L^p(\mathbb{T}^d)}\leq C\|\nabla v\|_{L^p(\mathbb{T}^d)}.
		\end{equation}
	\end{lemma}
\begin{proof}
	By direct computation, one has
	\begin{equation}\label{b3}
	\partial\left(\frac{(\rho v)^\varepsilon}{\rho ^\varepsilon}\right)=\frac{\partial (\rho v)^\varepsilon-v\partial \rho ^\varepsilon}{\rho^\varepsilon}-\frac{\left((\rho v)^\varepsilon-\rho ^\varepsilon v\right)\partial \rho ^\varepsilon}{(\rho ^\varepsilon)^2}:=I_1+I_2.
	\end{equation}	
	Let $B(x,\varepsilon)=\{y\in \mathbb{T}^d; |x-y|<\varepsilon\}$, then Using the H\"older's inequality, we have
	\begin{equation}\label{b4}
	\begin{aligned}
	|I_1|&\leq C|\int \rho(y)\left(v(y)-v(x)\right)\nabla_x\eta_\varepsilon(x-y)dy|\\
	&\leq C\|\rho \|_{L^\infty}|\int_{\mathbb{R}^d}|v(y)-v(x)|\frac{1}{\varepsilon^d}\nabla \eta (\frac{x-y}{^\varepsilon})\frac{1}{^\varepsilon}dy|\\
	&\leq C\left(\frac{1}{\varepsilon^d}\int _{B(x,\varepsilon)}\frac{|v(y)-v(x)|^p}{\varepsilon^p}dy\right)^{\frac{1}{p}}.\\
	\end{aligned}
	\end{equation}	
Then using the mean value theorem, one can obtain
	\begin{equation}\label{b5}
	\begin{aligned}
	\frac{1}{\varepsilon^d}\int_{B(x,\varepsilon)}\frac{|v(y)-v(x)|^p}{\varepsilon^p}dy&\leq C\frac{1}{\varepsilon^d}\int_{B(x,\varepsilon)}\int_0^1 |\nabla v(x+(y-x)s)|^p\frac{|y-x|^p}{\varepsilon^p}dsdy\\
	&\leq C \int_0^1\int_{B(0,1)}|\nabla v(x+s\varepsilon \omega)|^pd\omega ds\\
	&\leq C\int _{\mathbb{R}^d}|\nabla v(x-z)|^p\int_0^1\frac{\bf{1}_{(B£¨(0,\varepsilon s)}(z)}{(\varepsilon s)^d} ds dz\\
	&=(|\nabla v|^p *J_\varepsilon) (x),
	\end{aligned}
	\end{equation}	
	where $J_\varepsilon(z)=\int_0^1 \frac{\bf{1}_{(B£¨(0,\varepsilon s)}(z)}{(\varepsilon s)^d} ds\geq 0$ and it's easy to check that $\int_{\mathbb{R}^d} J_\varepsilon dz={\text \ measure\  of\  }(B(0,1))$.
	Next, to estimate $I_2$, due to the H\"older's inequality, one deduces
	\begin{equation}\label{b6}
	\begin{aligned}
	|I_2|&=|\int \rho(y)\left(v(y)-v(x)\right)\eta_\varepsilon(x-y)dy \frac{\int \rho(y)\nabla_x \eta_\varepsilon(x-y)dy}{\left(\int \rho(y)\eta_\varepsilon(x-y)dy\right)^2}|\\
	&\leq C\|\rho \|_{L^\infty} ^2\int_{B(x,\varepsilon)} |v(y)-v(x)|\frac{1}{\varepsilon^d}dy\int_{B(x,\varepsilon)} \frac{1}{\varepsilon^d}|\nabla \eta(\frac{x-y}{\varepsilon})|\frac{1}{\varepsilon} dy\\
	&\leq C \left(\frac{1}{\varepsilon^d}\int _{B(x,\varepsilon)}\frac{|v(y)-v(x)|^p}{\varepsilon^p}dy\right)^{\frac{1}{p}}.
	\end{aligned}
	\end{equation}
	Therefore, by the same arguments as in \eqref{b5}, in combination with \eqref{b3}-\eqref{b6}, we have
	\begin{equation}\label{b7}
	|I_1|+|I_2|\leq C\left(\nabla v|^p *J_\varepsilon \right)^{\frac{1}{p}}.
	\end{equation}	
	Then from the Minkowski's inequality, we arrive at
	\begin{equation}
	\begin{aligned}
	\B\|\partial \left(\frac{(\rho v)^{\varepsilon}}{\rho ^\varepsilon}\right)\B\|_{L^p(\mathbb{T}^d)}&\leq C\|\left(|\nabla v|^p*J_\varepsilon\right)^{\frac{1}{p}}\|_{L^p}\\
	&\leq C\|\nabla v\|_{L^p}\|J_\varepsilon\|_{L^1}^{\frac{1}{p}}\leq  C\|\nabla v\|_{L^p}.
	\end{aligned}
	\end{equation}
	Then we have completed the proof of lemma \eqref{lem2.3}.
	\end{proof}
\section{Proof of Theorem \ref{the1.2} and Corollary \ref{the1.1}}
In this section, we first present the proof of Theorem \ref{the1.2}. Then,
making use of interpolation and the natural energy, we prove  Corollary \ref{the1.1} by the results of Theorem \ref{the1.2}.
 \begin{proof}[Proof of Theorem \ref{the1.2}]
 	Let $\phi(t)$ be a smooth function compactly supported in $(0,+\infty)$. 	Multiplying $\eqref{GNS}_2$ by $\left(\phi(t)\frac{(\rho v)^\varepsilon}{\rho^\varepsilon}\right)^\varepsilon$, then integrating it   over $(0,T)\times \mathbb{T}^d$, we have
\begin{equation}\label{c1} \begin{aligned}
 \int_0^T\int \phi(t)\f{(\rho v)^{\varepsilon}}{\rho^{\varepsilon}}\B[\partial_{t}(\rho v)^{\varepsilon}+ \Div(\rho v\otimes v)^{\varepsilon}+\nabla p(\rho)^\varepsilon-\Div(\nu(\rho)\mathbb{D}v)^\varepsilon-\nabla(\mu(\rho)\Div v)^\varepsilon\B]=0.
 \end{aligned}\end{equation}
We will rewrite every term of the last  equality  to pass the limit of $\varepsilon$. For the first term in \eqref{c1}, a  straightforward  calculation and $\eqref{GNS}_{1}$ yields that
\begin{equation}\label{c2}
\begin{aligned} \int_0^T\int \phi(t)\f{(\rho v)^{\varepsilon}}{\rho^{\varepsilon}} \partial_{t}\B(\rho v\B)^{\varepsilon}&= \int_0^T\int\phi(t)\left[\f12\partial_{t}(\f{|(\rho v)^{\varepsilon}|^{2}}{\rho^{\varepsilon}})+\f12\partial_{t}\rho^{\varepsilon}
\f{|(\rho v)^{\varepsilon}|^{2}}{(\rho^{\varepsilon})^{2}}\right]\\
&= \int_0^T\int\phi(t)\left[\f12\partial_{t}\B(\f{|(\rho v)^{\varepsilon}|^{2}}{\rho^{\varepsilon}}\B)- \f12\Div(\rho v)^{\varepsilon}
\f{|(\rho v)^{\varepsilon}|^{2}}{(\rho^{\varepsilon})^{2}}\right].\\
 \end{aligned}\end{equation}
For the second term in \eqref{c1}, by integration by parts, it gives that
\begin{equation}\label{c3}\begin{aligned}
  &\int_0^T\int\phi(t)\f{(\rho v)^{\varepsilon}}{\rho^{\varepsilon}}  \Div(\rho v\otimes v)^{\varepsilon}\\
  =&-\int_0^T\int\phi(t)\nabla\B(\f{(\rho v)^{\varepsilon}}{\rho^{\varepsilon}} \B) [(\rho v\otimes v)^{\varepsilon}-(\rho v)^{\varepsilon}\otimes v^{\varepsilon}]-\int_0^T\int\phi(t)\nabla\B(\f{(\rho v)^{\varepsilon}}{\rho^{\varepsilon}} \B)(\rho v)^{\varepsilon}\otimes v^{\varepsilon}.
 \end{aligned}\end{equation}
 For the second term on the right hand side of above equality \eqref{c3}, it follows from the integration by parts once again that
 \begin{equation}\label{c4}\begin{aligned}
&- \int_0^T\int\phi(t)\nabla\B(\f{(\rho v)^{\varepsilon}}{\rho^{\varepsilon}} \B)(\rho v)^{\varepsilon}\otimes v^{\varepsilon}\\=& \int_0^T\int\phi(t)\left( \Div v^{\varepsilon} \f{|(\rho v)^\varepsilon|^{2}}{\rho^{\varepsilon}}+\f12 \f{ v^{\varepsilon}}{\rho^{\varepsilon}}\nabla|(\rho v)^{\varepsilon}  |^{2}\right)\\
=& \int_0^T\int\phi(t)\left( \f12 \Div v^{\varepsilon} \f{|(\rho v)^\varepsilon|^{2}}{\rho^{\varepsilon}}-\f12 v^{\varepsilon}\nabla({\f{1}{\rho^{\varepsilon}}} )|(\rho v)^{\varepsilon}  |^{2}\right)\\
=& \f12\int_0^T\int\phi(t) \Div( \rho^{\varepsilon}v^{\varepsilon} ) \f{|(\rho v)^\varepsilon|^{2}}{(\rho^{\varepsilon})^{2}} \\
 =& \f12\int_0^T\int\phi(t) \Div\B[ \rho^{\varepsilon}v^{\varepsilon}-(\rho v)^{\varepsilon} \B] \f{|(\rho v)^\varepsilon|^{2}}{(\rho^{\varepsilon})^{2}} +
 \f12\int_0^T\int\phi(t) \Div (\rho v)^{\varepsilon}   \f{|( \rho v)^\varepsilon|^{2}}{(\rho^{\varepsilon})^{2}}\\
 =&- \int_0^T\int\phi(t)\B[ \rho^{\varepsilon}v^{\varepsilon}-(\rho v)^{\varepsilon} \B] \f{(\rho v)^\varepsilon}
 {\rho^{\varepsilon}}\nabla\f{(\rho v)^\varepsilon}{\rho^{\varepsilon}} +
 \f12\int_0^T\int\phi(t) \Div (\rho v)^{\varepsilon}   \f{|(\rho v)^\varepsilon|^{2}}{(\rho^{\varepsilon})^{2}}.
 \end{aligned}\end{equation}
Then inserting  \eqref{c4}  into \eqref{c3}, we  have
 \begin{equation}\label{c5}\begin{aligned}
  &\int_0^T\int\phi(t)\f{(\rho v)^{\varepsilon}}{\rho^{\varepsilon}}  \Div(\rho v\otimes v)^{\varepsilon}\\=&-\int_0^T\int\phi(t)\nabla\B(\f{(\rho v)^{\varepsilon}}{\rho^{\varepsilon}} \B) [(\rho v\otimes v)^{\varepsilon}-(\rho v)^{\varepsilon}\otimes v^{\varepsilon}]\\&-\int_0^T\int\phi(t)\B[ \rho^{\varepsilon}v^{\varepsilon}-(\rho v)^{\varepsilon} \B] \f{(\rho v)^\varepsilon}
  {\rho^{\varepsilon}}\nabla \f{(\rho v)^\varepsilon}{\rho^{\varepsilon}}+
 \f12\int_0^T\int\phi(t) \Div (\rho v)^{\varepsilon}   \f{|(\rho v)^\varepsilon|^{2}}{(\rho^{\varepsilon})^{2}}.
   \end{aligned}\end{equation}
   For the pressure term in \eqref{c1}, together with the integration by parts, one has
    \begin{equation}\label{c6}
    \begin{aligned}
 &\int_0^T\int\phi(t)\f{(\rho v)^{\varepsilon}}{\rho^{\varepsilon}} \nabla(p(\rho))^{\varepsilon}\\
 = &\int_0^T\int\phi(t)\frac{(\rho v)^\varepsilon}{\rho^\varepsilon}\nabla \left[(p(\rho))^\varepsilon-p(\rho^\varepsilon)\right]+\int_0^T\int\phi(t)  \f{(\rho v)^{\varepsilon}}{\rho^{\varepsilon}} \nabla  p(\rho^{\varepsilon})\\
 =&-\int_0^T\int\phi(t)\Div\B[\f{(\rho v)^{\varepsilon}}{\rho^{\varepsilon}} \B] [(p(\rho))^{\varepsilon}- p(\rho^{\varepsilon}) ]+\int_0^T\int\phi(t)  \f{(\rho v)^{\varepsilon}}{\rho^{\varepsilon}} \nabla  p(\rho^{\varepsilon}).
\end{aligned} \end{equation}
Using the mass equation $\eqref{GNS}_1$, the second term on the right hand-side of \eqref{c6} can be rewritten as
\begin{equation}\label{c7}
\begin{aligned}
\int_0^T\int\phi(t) \frac{(\rho v)^\varepsilon}{\rho ^\varepsilon}\nabla p(\rho^\varepsilon)&=\int_0^T\int\phi(t) (\rho v)^\varepsilon\nabla \int_{1}^{\rho ^\varepsilon}
\frac{p^{'}(z)}{z}dzdxdt\\
&=\int_0^T\int\phi(t) \partial_t\rho^\varepsilon\left[\frac{p(\rho^\varepsilon)}{\rho ^\varepsilon}+\int_{1}^{\rho^\varepsilon}\frac{p(z)}{z^2}dz\right]dxdt\\
&=\int_0^T\int\phi(t) \partial_t P(\rho^\varepsilon),
\end{aligned}
\end{equation}
where
$P(\rho^\varepsilon)=\rho^\varepsilon\int_{1}^{\rho^\varepsilon}\frac{p(z)}{z^2}dz.$\\
Finally, for the viscous terms in \eqref{c1}, using the integration by parts, we have
 \begin{equation}\label{c8}\begin{aligned} &-\int_0^T\int\phi(t)\f{(\rho v)^{\varepsilon}}{\rho^{\varepsilon}} \Div(\nu(\rho)\mathbb{D}v)^{\varepsilon}\\
 	= 	&\int_0^T\int\phi(t)\left(-\Div (\nu(\rho)\mathbb{D}v) ^\varepsilon v^\varepsilon -\Div(\nu(\rho)\mathbb{D}v)^\varepsilon
 \f{(\rho v)^{\varepsilon}-\rho^{\varepsilon} v^{\varepsilon}}{\rho^{\varepsilon}}\right),\\
\end{aligned}\end{equation}
 \text{and}
 \begin{equation}\label{c8-1}\begin{aligned}
&-\int_0^T\int\phi(t)\f{(\rho v)^{\varepsilon}}{\rho^{\varepsilon}}\nabla(\mu(\rho)\Div v)^\varepsilon\\
=& \int_0^T\int\phi(t)\left( -\nabla  (\mu(\rho)\Div v)^\varepsilon  v^{\varepsilon}-\nabla(\mu(\rho)\Div v)^{\varepsilon}
 \f{(\rho v)^{\varepsilon}-\rho^{\varepsilon} v^{\varepsilon}}{\rho^{\varepsilon}}\right).
\end{aligned}  \end{equation}
Then substituting \eqref{c2}, \eqref{c5}-\eqref{c8-1} into \eqref{c1}, we see that
 \begin{equation}\label{c9}
 \begin{aligned}
&\int_0^T\int\phi(t)\partial_{t}\left(\f12\f{|(\rho v)^{\varepsilon}|^{2}}{\rho^{\varepsilon}}+P(\rho^\varepsilon)\right) -\int_0^T\int\phi(t)\left( \Div\left(\nu(\rho)\mathbb{D}v\right)^{\varepsilon}  v^{\varepsilon}+ \nabla(\mu(\rho)\Div v)^{\varepsilon}  v^{\varepsilon}\right)\\
=
 &\int_0^T\int\phi(t) \Div(\nu(\rho)\mathbb{D}v)^{\varepsilon}
 \f{(\rho v)^{\varepsilon}-\rho^{\varepsilon} v^{\varepsilon}}{\rho^{\varepsilon}}+\int_0^T\int\phi(t)\nabla(\mu(\rho)\Div v)^{\varepsilon}
 \f{(\rho v)^{\varepsilon}-\rho^{\varepsilon} v^{\varepsilon}}{\rho^{\varepsilon}}
\\&\ \ \ +\int_0^T\int\phi(t)\Div\B[\f{(\rho v)^{\varepsilon}}{\rho^{\varepsilon}} \B] [(p(\rho))^{\varepsilon}- p(\rho^{\varepsilon}) ]\\
+&\int_0^T\int\phi(t)\nabla\B(\f{(\rho v)^{\varepsilon}}{\rho^{\varepsilon}} \B) [(\rho v\otimes v)^{\varepsilon}-(\rho v)^{\varepsilon}\otimes v^{\varepsilon}] +\int_0^T\int\phi(t)\B[ \rho^{\varepsilon}v^{\varepsilon}-(\rho v)^{\varepsilon} \B] \f{(\rho v)^\varepsilon}
{\rho^{\varepsilon}}\nabla\f{(\rho v)^\varepsilon}{\rho^{\varepsilon}}.\\
  \end{aligned}\end{equation}
  Next, we need to prove that the terms on the right hand-side of \eqref{c9} tend to zero as $\varepsilon\rightarrow 0$.

Firstly, it follows from Lemma \ref{lem2.1} and Lemma  \ref{lem2.2} that
\begin{equation}\label{c10}\begin{aligned}
&\|\Div(\nu(\rho)\mathbb{D}v)^\varepsilon\|_{L^{2}(L^{2})}\leq C\varepsilon^{-1}\| \nu(\rho)\mathbb{D}v \|_{L^{2}(L^{2})}\leq C\varepsilon^{-1}\|\nabla v\|_{L^2(L^2)},\\
&\limsup_{\varepsilon\rightarrow0}\varepsilon\|\Div(\nu(\rho)\mathbb{D}v)^\varepsilon\|_{L^{2}(L^{2})}=0,\\
&	\|(\rho v)^\varepsilon-\rho^\varepsilon v^\varepsilon\|_{L^2(L^2)}\leq C\varepsilon \|\rho\|_{L^{\infty}(L^\infty)}\|v\|_{L^{2}(W^{1,2})}.
  \end{aligned}\end{equation}
Moreover, due to the H\"older's inequality, we can obtain that
\begin{equation}\label{c11}\begin{aligned}
&\B|\int_0^T\int\phi(t)\Div(\nu(\rho)\mathbb{D}v)^\varepsilon
 \f{(\rho v)^{\varepsilon}-\rho^{\varepsilon} v^{\varepsilon}}{\rho^{\varepsilon}}\B|\\
\leq& C\|\Div(\nu(\rho)\mathbb{D}v)^\varepsilon\|_{L^{2}(L^{2})}\|
 \f{(\rho v)^{\varepsilon}-\rho^{\varepsilon} v^{\varepsilon}}{\rho^{\varepsilon}}\|_{L^{2}(L^{2})}\\
\leq& C\varepsilon\|\Div(\nu(\rho)\mathbb{D}v)^\varepsilon\|_{L^{2}(L^{2})}\|\rho \|_{L^{\infty}L^{\infty}} \|v\|_{L^{2}(W^{1,2})}. \end{aligned}\end{equation}
 As a consequence, in combination with \eqref{c10} and \eqref{c11}, we have
$$ \limsup_{\varepsilon\rightarrow0}\B|\int_0^T\int\phi(t)\Div(\nu(\rho)\mathbb{D}v)^\varepsilon
 \f{(\rho v)^{\varepsilon}-\rho^{\varepsilon} v^{\varepsilon}}{\rho^{\varepsilon}}\B|=0.
$$
Likewise, there also holds
\begin{equation}\label{c12}
\begin{aligned}
\limsup_{\varepsilon\rightarrow0}\B|\int_0^T\int\phi(t)\nabla(\mu(\rho)\Div v)^{\varepsilon}
 \f{(\rho v)^{\varepsilon}-\rho^{\varepsilon} v^{\varepsilon}}{\rho^{\varepsilon}}\B|=0.
\end{aligned}\end{equation}
Next, by means of  the triangle inequality, the H\"older inequality and Lemma \ref{lem2.3}, we obtain
\begin{equation}\label{c13}\begin{aligned}  &\int_0^T\int\phi(t)\Div\B[\f{(\rho v)^{\varepsilon}}{\rho^{\varepsilon}} \B] [(p(\rho))^{\varepsilon}- p(\rho^{\varepsilon}) ]\\
  \leq& \int_0^T\int\phi(t)\B|\Div\B[\f{(\rho v)^{\varepsilon}}{\rho^{\varepsilon}} \B] \B| |(p(\rho))^{\varepsilon}- p(\rho ) |+ \int_0^T\int\phi(t)\B|\Div\B[\f{(\rho v)^{\varepsilon}}{\rho^{\varepsilon}} \B]\B| | p(\rho) - p(\rho^{\varepsilon}) |\\
  \leq& C\|\Div\B[\f{(\rho v)^{\varepsilon}}{\rho^{\varepsilon}} \B]\|_{L^{2}(L^{2})}\B(\|(p(\rho))^{\varepsilon}- p(\rho )\|_{L^{2}(L^{2})}+\| p(\rho) - p(\rho^{\varepsilon} )\|_{L^{2}(L^{2})}\B)
  \\
 \leq& C\|\nabla v\|_{L^{2}(L^{2})}  \B(\| (p(\rho))^{\varepsilon} - p(\rho )\|_{L^{2}(L^{2})}+\|p'\|_{L^\infty(L^\infty)}\|  \rho  -  \rho^{\varepsilon} \|_{L^{2}(L^{2})}\B),
  \end{aligned}\end{equation}
  which implies that
  $$
  \limsup_{\varepsilon\rightarrow0}\int_0^T\int\phi(t)\Div\left(\f{(\rho v)^{\varepsilon}}{\rho^{\varepsilon}} \right) \big((p(\rho))^{\varepsilon}- p(\rho^{\varepsilon}) \big)=0.$$
At this stage,  it is enough to show
\begin{equation}\label{c14}\begin{aligned}
\limsup_{\varepsilon\rightarrow0}\int_0^T\int&\phi(t)\nabla\B(\f{(\rho v)^{\varepsilon}}{\rho^{\varepsilon}} \B) [(\rho v\otimes v)^{\varepsilon}-(\rho v)^{\varepsilon}\otimes v^{\varepsilon}]\\
 &+\limsup_{\varepsilon\rightarrow0}\int_0^T\int\phi(t)\B[ \rho^{\varepsilon}v^{\varepsilon}-(\rho v)^{\varepsilon} \B] \f{(\rho v)^\varepsilon}
 {\rho^{\varepsilon}}\nabla\f{(\rho v)^\varepsilon}{\rho^{\varepsilon}} =0,
  \end{aligned}\end{equation}
under the hypothesis
\be\label{unified}
v\in L^{\frac{2p}{p-1}}(L^{\frac{2q}{q-1}})~~ \text{ and }~~\nabla v\in L^p(L^q).
\ee
To do this, applying Lemma \ref{lem2.2}, we obtain that
\begin{equation}\label{c17} \begin{aligned}
&\|(\rho v\otimes v)^{\varepsilon}-(\rho v)^{\varepsilon}\otimes v^{\varepsilon} \|_{L^{\f{2p}{p+1}}(L^{\f{2q}{q+1}})}\leq C\varepsilon\|v\|_{L^p(W^{1,q})} \|\rho v\|_{L^{\frac{2p}{p-1}}(L^{\frac{2q}{q-1}})},\\
&\B\|\nabla\B(\f{(\rho v)^{\varepsilon}}{\rho^{\varepsilon}} \B)\B\|_{L^{\f{2p}{p-1}}(L^{\f{2q}{q-1}})}\leq C\varepsilon^{-1}\|\rho v\|_{L^{\f{2p}{p-1}}(L^{\f{2q}{q-1}})},\\ &\limsup_{\varepsilon\rightarrow0}\varepsilon\B\|\nabla\B(\f{(\rho v)^{\varepsilon}}{\rho^{\varepsilon}} \B)\B\|_{L^{\f{2p}{p-1}}(L^{\f{2q}{q-1}})}=0.
\end{aligned}\end{equation}
Using the H\"older's inequality and Lemma \ref{lem2.1}, we find
  \begin{equation}\label{c18}\begin{aligned}
 &\B|\int_0^T\int\phi(t)\nabla\B(\f{(\rho v)^{\varepsilon}}{\rho^{\varepsilon}} \B) [(\rho v\otimes v)^{\varepsilon}-(\rho v)^{\varepsilon}\otimes v^{\varepsilon}]\B|\\
\leq & C\B\|\nabla\B(\f{(\rho v)^{\varepsilon}}{\rho^{\varepsilon}} \B)\B\|_{L^{\f{2p}{p-1}}(L^{\f{2q}{q-1}})}\|(\rho v\otimes v)^{\varepsilon}-(\rho v)^{\varepsilon}\otimes v^{\varepsilon} \|_{L^{\f{2p}{p+1}}(L^{\f{2q}{q+1}})}\\
\leq & C\varepsilon\B\|\nabla\B(\f{(\rho v)^{\varepsilon}}{\rho^{\varepsilon}} \B)\B\|_{L^{\f{2p}{p-1}}(L^{\f{2q}{q-1}})}\|v\|_{L^p(W^{1,q})} \|\rho v\|_{L^{\frac{2p}{p-1}}(L^{\frac{2q}{q-1}})}\\
\leq & C\varepsilon\B\|\nabla\B(\f{(\rho v)^{\varepsilon}}{\rho^{\varepsilon}} \B)\B\|_{L^{\f{2p}{p-1}}(L^{\f{2q}{q-1}})}\|v\|_{L^p(W^{1,q})} \| v\|_{L^{\frac{2p}{p-1}}(L^{\frac{2q}{q-1}})}\\
\leq & C\varepsilon\B\|\nabla\B(\f{(\rho v)^{\varepsilon}}{\rho^{\varepsilon}} \B)\B\|_{L^{\f{2p}{p-1}}(L^{\f{2q}{q-1}})},
 \end{aligned}\end{equation}
 which in turn gives
$$
\limsup_{\varepsilon\rightarrow0}\B|\int_0^T\int\phi(t)\nabla\B(\f{(\rho v)^{\varepsilon}}{\rho^{\varepsilon}} \B) [(\rho v\otimes v)^{\varepsilon}-(\rho v)^{\varepsilon}\otimes v^{\varepsilon}]\B|=0.$$
Now, we turn our attentions to the term $\int_0^T\int\phi(t)\B[ \rho^{\varepsilon}v^{\varepsilon}-(\rho v)^{\varepsilon} \B] \f{(\rho v)^\varepsilon}
{\rho^{\varepsilon}}\nabla\f{(\rho v)^\varepsilon}{\rho^{\varepsilon}}$. Since  $\rho v\in L^{\f{2p}{p-1}}(L^{\f{2q}{q-1}}) $, we derive from Lemma \ref{lem2.1}
that
\begin{equation}\label{1.3.15}
\limsup_{\varepsilon\rightarrow0}\varepsilon\B\|\nabla \left(\f{(\rho v)^\varepsilon}{\rho^{\varepsilon}}\right)\B\|_{L^{\f{2p}{p-1}}(L^{\f{2q}{q-1}})}
=0.
\end{equation}
In addition, we conclude from Lemma \ref{lem2.2} that
\begin{equation}\label{1.3.16}
\|\rho^{\varepsilon}v^{\varepsilon}-(\rho v)^{\varepsilon}\|_{L^{p}(L^{q})} \leq C\varepsilon\|v\|_{L^{p}(W^{1,q})} \|\rho\|_{L^{\infty}(L^{\infty})}.
\end{equation}
Then, together with the H\"older's inequality  and \eqref{1.3.16}, we have,
\begin{equation}\label{1.3.17} \begin{aligned}
&\B|\int_0^T\int\phi(t)\B[ \rho^{\varepsilon}v^{\varepsilon}-(\rho v)^{\varepsilon} \B] \f{(\rho v)^\varepsilon}
{\rho^{\varepsilon}}\nabla \left(\f{(\rho v)^\varepsilon}{\rho^{\varepsilon}}\right)\B| \\
\leq&C  \|\rho^{\varepsilon}v^{\varepsilon}-(\rho v)^{\varepsilon}\|_{L^{p}(L^{q})}\B\|\f{(\rho v)^\varepsilon}
{\rho^{\varepsilon}}\B\|_{L^{\frac{2p}{p-1}}(L^{\frac{2q}{q-1}})} \B\|\nabla\left(\f{(\rho v)^\varepsilon}{\rho^{\varepsilon}}\right)\B\|_{L^{\f{2p}{p-1}}(L^{\f{2q}{q-1}})} \\
\leq& C\|v\|_{L^{p}(W^{1,q})} \|\rho\|_{L^{\infty}(L^{\infty})} \|v\|_{L^{\frac{2p}{p-1}}(L^{\frac{2q}{q-1}})}\varepsilon \B\|\nabla\left(\f{(\rho v)^\varepsilon}{\rho^{\varepsilon}}\right)\B
\|_{L^{\f{2p}{p-1}}
(L^{\f{2q}{q-1}})}\\
\leq& C\varepsilon \B\|\nabla\left(\f{(\rho v)^\varepsilon}{\rho^{\varepsilon}}\right)\B
\|_{L^{\f{2p}{p-1}}
	(L^{\f{2q}{q-1}})},\\
\end{aligned}\end{equation}
which together with \eqref{1.3.15} yields that
$$
\limsup_{\varepsilon\rightarrow0}
\B|\int_0^T\int\phi(t)\B[ \rho^{\varepsilon}v^{\varepsilon}-(\rho v)^{\varepsilon} \B] \f{(\rho v)^\varepsilon}
{\rho^{\varepsilon}}\nabla \f{(\rho v)^\varepsilon}{\rho^{\varepsilon}}\B|=0.
$$
Collecting  all the above estimates, using the integration by parts with respect to $t$ and letting $\varepsilon\rightarrow 0$, for any $\phi(t)\in \mathcal{D} (0,T) $, we have
\begin{equation}\label{c27}
	\begin{aligned}
		&-\int_0^T\int\partial_{t}\phi(t)\left(\f12\rho v^2+P(\rho)\right) + \int_0^T\int\phi(t) \left(\nu(\rho)|\mathbb{D} v|^2+\mu(\rho)|\Div v|^2 \right)=0.\\
\end{aligned}\end{equation}

The next objective is to get the energy equality up to the initial time $t=0$. First, we claim that for any $t_0\geq 0$,
$$\lim\limits_{t\rightarrow t_0^+}\|\sqrt{\rho} v(t)\|_{L^2(\mathbb{T}^d)}=\|\sqrt{\rho}v(t_0)\|_{L^2(\mathbb{T}^d)}\ \ \text{and}\ \ \lim\limits_{t\rightarrow t_0^+}\|P(\rho)(t)\|_{L^1(\mathbb{T}^d)}=\|P(\rho)(t_0)\|_{L^1(\mathbb{T}^d)}.$$
In fact, by the
energy estimates \eqref{1.13} and  the weak continuity of $\rho$ and  $\rho v$  in \eqref{1.14} and \eqref{1.15},  we have
\begin{equation}\label{3.22}
	\begin{aligned}
	&	\rho v\in L^\infty (0,T; L^2(\mathbb{T}^d))\cap H^1(0,T; W^{-1,1}(\mathbb{T}^d))\hookrightarrow C([0,T]; L^2_{weak}(\mathbb{T}^d)),\\
	&	\text{and}\\
	&\sqrt{\rho}\in L^\infty(0,T;L^\infty(\mathbb{T}^d))\cap H^1(0,T; H^{-1}(\mathbb{T}^d))\hookrightarrow C([0,T]; L^l_{weak}(\mathbb{T}^d)),\ \text{for\ any}\ l\in (1,+\infty)\\
	\end{aligned}
\end{equation}
Due to the convexity of $\rho \mapsto P(\rho)$, we have
\begin{equation}\label{3.23}
	\begin{aligned}
	\int_{\mathbb{T}^d} P(\rho)(t_0)\leq \underline{\lim\limits_{t\rightarrow t_0^+}}	\int_{\mathbb{T}^d} P(\rho)(t),\ \ \text{for\ any}\ t_0\geq 0.
	\end{aligned}
\end{equation}
Meanwhile, using the natural energy \eqref{1.13},\eqref{1.16}, \eqref{3.22}and \eqref{3.23}, we have
\begin{equation}\label{d17}
	\begin{aligned}
		0&\leq \overline{\lim_{t\rightarrow 0^+}}\int |\sqrt{\rho} v-\sqrt{\rho_0}v_0|^2 dx\\
		&=2\overline{\lim_{t\rightarrow 0^+}}\left(\int \left(\f{1}{2}\rho |v|^2 +P(\rho) \right)dx-\int\left(\f{1}{2}\rho_0 |v_0|^2+P(\rho_0) \right)dx\right)\\
		&\ \ \ +2\overline{\lim_{t\rightarrow 0^+}}\left(\int\sqrt{\rho_0}v_0\left(\sqrt{\rho_0}v_0-\sqrt{\rho} v\right)dx+\int \left(P(\rho_0) -P(\rho)\right)dx\right)\\
		&\leq 2\overline{\lim_{t\rightarrow 0^+}}\int \sqrt{\rho_0}v_0\left(\sqrt{\rho_0}v_0-\sqrt{\rho}v\right)dx\\
		&=0,
	\end{aligned}
\end{equation}
where the last equality sign comes from
\begin{equation}\label{3.25}
	\begin{aligned}
		&2\overline{\lim_{t\rightarrow 0^+}}\int \sqrt{\rho_0}v_0\left(\sqrt{\rho_0}v_0-\sqrt{\rho}v\right)dx\\
		=&2\overline{\lim_{t\rightarrow 0^+}}\int \frac{\sqrt{\rho_0}v_0}{\sqrt{\rho}}\left(\sqrt{\rho}\sqrt{\rho_0}v_0-\rho v\right)dx\\
		\leq &2\overline{\lim_{t\rightarrow 0^+}}\int \frac{\sqrt{\rho_0}v_0}{\sqrt{\rho}}\left(\sqrt{\rho}\sqrt{\rho_0}v_0-\rho_0 v_0\right)dx+2\overline{\lim_{t\rightarrow 0^+}}\int \frac{\sqrt{\rho_0}v_0}{\sqrt{\rho}}\left({\rho_0}v_0-\rho v\right)dx\\
		\leq &2\overline{\lim_{t\rightarrow 0^+}}\int (\sqrt{\rho_0}v_0)^2\left(\sqrt{\rho}-\sqrt{\rho_0}\right)dx+2\overline{\lim_{t\rightarrow 0^+}}\int\sqrt{\rho_0} v_0\left({\rho_0}v_0-\rho v\right)dx\\
		=&0,
	\end{aligned}
\end{equation}
where we used \eqref{1.13}, \eqref{3.22} and $\sqrt{\rho_0}v_0\in L^{2+\delta}$ for any $\delta>0$.
Then we have
\begin{equation}\label{d18}
	\sqrt{\rho} v(t)\rightarrow \sqrt{\rho }v(0)\ \ strongly\ in\ L^2(\mathbb{T}^d)\ as\ t\rightarrow 0^+.
\end{equation}
Similarly, one has the right temporal continuity  of $\sqrt{\rho}v$ in $L^2(\mathbb{T}^d)$, hence, for any $t_0\geq 0$, we infer that
\begin{equation}\label{d19}
	\sqrt{\rho} v(t)\rightarrow \sqrt{\rho }v(t_0)\ \ strongly\ in\ L^2(\mathbb{T}^d)\ as\ t\rightarrow t_0^+.
\end{equation}
Next, it follows from \eqref{1.16} that
\begin{equation}\label{3.28}
	\overline{\lim_{t\rightarrow t_0^+}} \mathcal{E}(t)\leq \mathcal{E}(t_0).
\end{equation}
This and  \eqref{d19} imply
\begin{equation}\label{3.29}
\overline{	\lim\limits_{t\rightarrow t_0^+}}\|P(\rho)(t)\|_{L^1(\mathbb{T}^d)}\leq \|P(\rho)(t_0)\|_{L^1(\mathbb{T}^d)}.
\end{equation}
Notice from \eqref{1.13} and the mass equation $\eqref{GNS}_1$ that
\begin{equation}\label{3.30}
	P(\rho)\in L^\infty(0,T;L^\infty(\mathbb{T}^d))\cap H^1(0,T;H^{-1}(\mathbb{T}^d))\hookrightarrow C([0,T];L^2_{\text weak }(\mathbb{T}^d)).\end{equation}
Hence, \eqref{3.23}, \eqref{3.29} and \eqref{3.30} guarantee
\begin{equation}\label{3.31}\lim\limits_{t\rightarrow t_0^+}\|P(\rho)(t)\|_{L^1(\mathbb{T}^d)}=\|P(\rho)(t_0)\|_{L^1(\mathbb{T}^d)}.\end{equation}
Before we go any further, it should be noted that \eqref{c27} remains valid for function $\phi$ belonging to $W^{1,\infty}$ rather than $C^1$, then for any $t_0>0$, we redefine the test function $\phi$ as $\phi_\tau$ for some positive $\tau$ and $\alpha $ such that $\tau +\alpha <t_0$, that is
\begin{equation}
	\phi_\tau(t)=\left\{\begin{array}{lll}
		0, & 0\leq t\leq \tau,\\
		\f{t-\tau}{\alpha}, & \tau\leq t\leq \tau+\alpha,\\
		1, &\tau+\alpha \leq t\leq t_0,\\
		\f{t_0-t}{\alpha }, & t_0\leq t\leq t_0 +\alpha ,\\
		0, & t_0+\alpha \leq t.
	\end{array}\right.
\end{equation}
Then substituting this test function into \eqref{c27}, we arrive at
\begin{equation}
	\begin{aligned}
		-\int_\tau^{\tau+\alpha}\int& \f{1}{\alpha}\left(\f{1}{2}\rho v^2+P(\rho) \right)+\f{1}{\alpha}\int_{t_0}^{t_0+\alpha}\int 	\left(\f{1}{2}\rho v^2+P(\rho)\right)\\
		&+\int_{\tau}^{t_0+\alpha}\int \phi_\tau \left(\nu(\rho)|\mathbb{D} v|^2+\mu(\rho)|\Div v|^2 \right)=0.
	\end{aligned}	
\end{equation}
Taking $\alpha\rightarrow 0$ and using  the fact that $\int_0^t\int\left(\nu(\rho)|\mathbb{D} v|^2+\mu(\rho)|\Div v|^2 \right)$ is continuous with respect to $t$ and the  Lebesgue point Theorem, we deduce that
\begin{equation}
	\begin{aligned}
		-\int&\left(\f{1}{2}\rho v^2+P(\rho) \right)(\tau)dx+\int\left(\f{1}{2}\rho v^2+P(\rho) \right)(t_0)dx\\
		&+\int_\tau^{t_0}\int\left(\nu(\rho)|\mathbb{D} v|^2+\mu(\rho)|\Div v|^2 \right)=0.
	\end{aligned}
\end{equation}
Finally, letting $\tau\rightarrow 0$, using the continuity of $\int_0^t\int\left(\nu(\rho)|\mathbb{D} v|^2+\mu(\rho)|\Div v|^2 \right)$, \eqref{d19} and \eqref{3.31}, we can obtain
\begin{equation}\ba
	\int\left(\f{1}{2}\rho v^2+P(\rho) \right)(t_0)dx+\int_0^{t_0}\int&\left(\nu(\rho)|\mathbb{D} v|^2+\mu(\rho)|\Div v|^2 \right)dxds\\=&\int\left(\f{1}{2}\rho_0 v_0^2+P(\rho_0) \right)dx.
	\ea\end{equation}
Then we have completed the proof of Theorem \ref{the1.2}.

\end{proof}
We are in a position to prove Corollary \ref{the1.1}.
 \begin{proof}[Proof of Corollary \ref{the1.1}]
(1) The natural energy gives $v \in L^{2}(0,T;H^{1}(\mathbb{T}^3))$. Choosing $p=q=2$ in \eqref{u}, we 	immediately prove that the condition $v\in L^{4}(0,T;L^{4}(\mathbb{T}^3))$ yields energy equality.

It is worth remarking that
the rest proof  in \eqref{vei} can be reduced to this special case.
Next, we first deal with the case \eqref{vei} in Corollary
\ref{the1.1} with $q\geq 4$ and $\frac{2}{p}+\frac{2}{q}=1$.
The  Gagliardo-Nirenberg   inequality  guarantees that
\begin{equation}
\begin{aligned}
\|v\|_{L^{4}(0,T;L^{4}(\mathbb{T}^3) )}  \leq& C\|v\|_{L^{\infty}\left(0,T;L^{2}(\mathbb{T}^{3})\right)}^{\frac{(q-4)}{2 q-4}}\|v\|_{L^{p}(0,T;L^q(\mathbb{T}^3))}^{\frac{q}{2q-4}}\leq C.
\end{aligned}
\end{equation}
From the result just proved, we obtain  energy equality via \eqref{vei}  with $q\geq 4$. \\
Then we consider \eqref{vei} with $3<q<4$ and $\frac{1}{p}+\frac{3}{q}=1$.
Using the Gagliardo-Nirenberg   inequality again, we know that
\begin{equation}
\begin{aligned}
\|v\|_{L^{4}(0,T;L^{4}(\mathbb{T}^3))}
&\leq C\|v\|_{L^2(0,T;L^6(\mathbb{T}^3))}^{\frac{3(4-q)}{2(6-q)}}\|v\|_{L^p(0,T;L^q(\mathbb{T}^3))}^{\frac{q}{2(6-q)}}\\
	&\leq C\left(\|\nabla v\|_{L^2(0,T;L^2(\mathbb{T}^3))}+\|v\|_{L^\infty(0,T;L^2(\mathbb{T}^3))}\right)^{\frac{3(4-q)}{2(6-q)}}\|v\|_{L^p(0,T;L^q(\mathbb{T}^3))}^{\frac{q}{2(6-q)}}\leq C.
\end{aligned}
\end{equation}
 We finish the proof of \eqref{vei}.

(2) Now, we focus on the proof of \eqref{tivei}. Indeed, note that $  v\in L^{p}(0,T;W^{1,q}(\mathbb{T}^{3}))$, therefore, according to Theorem \ref{the1.2},
it suffices to derive  $v\in L^{\f{2p}{p-1}}(0,T;L^{\f{2q}{q-1}}(\mathbb{T}^3)) $   from  \eqref{tivei}.  For $q\geq \frac{9}{5}$, by the Gagliardo-Nirenberg   inequality, we get
 \begin{equation}\label{c19}
 \begin{aligned}
  \|v\|_{L^{\f{2q}{q-1}}(\mathbb{T}^3)}\leq C\|v\|_{L^{2}(\mathbb{T}^3)}^{\f{5q-9}{5q-6}}\|\nabla v\|_{L^{q}(\mathbb{T}^3)}^{\f{3}{5q-6}}.
  \end{aligned}\end{equation}
  Thanks to $\frac{1}{p}+\frac{6}{5 q}=1$, we further infer that
\begin{equation}\label{c20}
\begin{aligned}
  \|v\|_{L^{\f{2p}{p-1}}(0,T;L^{\f{2q}{q-1}}(\mathbb{T}^3))}\leq C \|v\|_{L^{\infty}(0,T;L^{2}(\mathbb{T}^3))}^{\f{5q-9}{5q-6}}\|\nabla v\|_{L^{p}(0,T;L^{q}(\mathbb{T}^3))}^{\frac{3}{5q-6}}\leq C.
  \end{aligned}\end{equation}
In light of Theorem \ref{the1.2}, we have proved \eqref{tivei} for $q\geq \frac{9}{5}$.

Finally,  for $\frac{3}{2}<q<\frac{9}{5}$, it follows the Gagliardo-Nirenberg inequality that
\begin{equation}\label{c21}
\begin{aligned}
\|v\|_{L^\frac{2q}{q-1}(\mathbb{T}^3)}\leq C\|v\|_{L^6(\mathbb{T}^3)}^{\frac{9-5q}{6-3q}}\|\nabla v\|_{L^q(\mathbb{T}^3)}^{\frac{2q-3}{6-3q}},
\end{aligned}
\end{equation}
Thanks to $\frac{1}{p}+\frac{3}{q}=2$, we further have
\begin{equation}
\begin{aligned}
\|v\|_{L^{\frac{2p}{p-1}}(0,T;L^{\frac{2q}{q-1}}(\mathbb{T}^3))}&\leq C\|v\|_{L^2(0,T;L^6(\mathbb{T}^3))}^{\frac{9-5q}{6-3q}}\|\nabla v\|_{L^p(0,T;L^q(\mathbb{T}^3))}^{\frac{2q-3}{6-3q}}\\
&\leq \left(\|\nabla v\|_{L^2(0,T;L^2(\mathbb{T}^3))}+\|v\|_{L^\infty(0,T;L^2(\mathbb{T}^3))}\right)^{\frac{9-5q}{6-3q}}\|\nabla v\|_{L^p(0,T;L^q(\mathbb{T}^3))}^{\frac{2q-3}{6-3q}},
\end{aligned}
\end{equation}
then we conclude the desired result from Theorem \ref{the1.2}. The proof of this Corollary is completed.
 \end{proof}

\section*{Acknowledgement}

 Wang was partially supported by  the National Natural
 Science Foundation of China under grant (No. 11971446, No. 12071113   and  No.  11601492). Ye was partially supported by the National Natural Science Foundation of China  under grant (No.11701145) and China Postdoctoral Science Foundation (No. 2020M672196).

\end{document}